\documentclass{amsart}
\usepackage[utf8]{inputenc}
\usepackage{biblatex}
\usepackage{enumitem}
\usepackage{varwidth}
\usepackage{verbatim}

\newenvironment{centerverbatim}{%
  \par
  \centering
  \varwidth{\linewidth}%
  \verbatim
}{%
  \endverbatim
  \endvarwidth
  \par
}
\usepackage{amssymb}
\usepackage{amsmath}
\usepackage{mathabx}
\usepackage{mathtools}

\usepackage{graphicx}
\bibliography{bib}

\usepackage{amsmath}
\usepackage{amssymb}
\usepackage{amsthm}

\usepackage[final]{pdfpages}

\newcommand{\ModCmp}[3]{#1 \equiv #2 \;(\bmod\; #3)}

\DeclareMathOperator{\Cyc}{Cyclic}
\DeclareMathOperator{\dis}{d}
\DeclareMathOperator{\Exoo}{Exoo}
\DeclareMathOperator{\e}{e}

\theoremstyle{plain}
\newtheorem{prop}{Proposition}[section]

\newtheorem{thm}{Theorem}[section]
\newtheorem{lem}{Lemma}[section]
\newtheorem{conj}{Conjecture}[section]

\theoremstyle{definition}
\newtheorem{eg}{Example}[section]
\newtheorem{defn}{Definition}[section]

\title{Study of Exoo's lower bound for Ramsey number \boldmath{$R(5,5)$}}
\author[A.Q, Y.J, M.S, V.Y, L.G]{Lachlan Ge, Yasiru Jayasooriya, Alex Qiu, Michael Sun, Victor Yuan, }
\date{November 2022}

\begin{document}

\begin{abstract}
    We review Exoo's 1989 paper, which demonstrates that a lower bound for the Ramsey number $R(5,5)$ is $43$. We provide an efficient way to verify the claims in the paper, adding detailed proofs. In particular, we replace the reference to computer verification by concise arguments. Using our understanding of the insight behind these proofs, we are also able to analyse variations of the graph constructions to obtain, for example, colourings of $K_{43}$ which have very few monochromatic $K_5$.
\end{abstract}
\maketitle

\section*{Introduction}
What is the smallest number of people that there can be in a group which guarantees that there will always be three people who all know each other or three people who all don't know each other? This is a well-known puzzle whose answer is six. \\

However, if we consider a slightly larger variation and ask ``What is the minimum number of people for which we can guarantee that there exists five people who all know each other or five people who all don't know each other?'' then we have a problem whose answer is currently unknown! The answer to this question is denoted by $R(5, 5)$, the Ramsey number $(5,5)$, and determining its value is an open problem. \\ 

There is an amusing quote from Paul Erdős regarding this problem:

\begin{quote}
Imagine an alien force, vastly more powerful than us landing on Earth and demanding the value
of $R(5, 5)$ or they will destroy our planet. In that case, we should marshal all our computers and
all our mathematicians and attempt to find the value. But suppose, instead, that they asked for
$R(6, 6)$, we should attempt to destroy the aliens.
\end{quote}

It is currently known that $R(5,5)$ is between $43$ (Geoffrey Exoo 1989) and $48$ (Vigleik Angeltveit and  Brendan D. McKay 2017) inclusive and there is evidence that leads to the following conjecture \cite{MR}:

\begin{conj}[McKay--Radsizowski]

$$R(5,5)=43.$$
\end{conj}

In 1989 Exoo \cite{Exoo} published a construction which showed:
\begin{thm}[Exoo]
$$R(5,5)\geq 43.$$
\end{thm}

Exoo's paper will be the focus of our study. The problem is represented using a \emph{graph}, where the vertices represent the people and the edges represent the relationship between pairs of people. For concreteness, we use a \emph{complete graph}, which is a graph where every pair of vertices is joined by an edge. We then colour every edge in this complete graph one of two colours (either red or blue) to represent whether each pair of people knows or does not know each other. \\

We use the usual notation of $K_n$ for a complete graph with $n$ vertices. In the following sections, we will define a \textit{red} $K_5$ to be a complete graph with $5$ vertices where all of its edges are red and define a \textit{blue} $K_5$ similarly for blue edges. Also, we will denote an edge in the graph between vertices $u, v$ by $\e(u, v)$. \\

In the first section we recall the definition of Exoo's graph from Exoo's original paper. In the other sections, we provide explicit proofs for the claims whose proofs were omitted in the original paper, making everything self contained. The highlight of this is showing in Section 5 that no blue $K_5$'s exist in Exoo's graph, which was left to computer verification in the original paper. \\

\textbf{Section 1:} Definition of Exoo's graph: We recall the contents of Exoo's paper and organise them into statements and add some diagrams.
\vspace{0.5cm}

\textbf{Section 2:} Helpful observations: We state and explain some general principles that have helped to save significant time and effort in calculations and casework.
\vspace{0.5cm}

\textbf{Section 3:} Red checking: We write a complete proof that there are no red $K_5$'s in Exoo's graph by reducing it to 3 cases.
\vspace{0.5cm}

\textbf{Sections 4 and 5:} We write a complete proof that there are no blue $K_5$'s in Exoo's graph.

\subsection{Acknowledgements} Special thanks to Professors McKay, Radsizowski and Exoo for their helpful comments and encouragement.\\

Contributions from Rhianna Kho, Callum Kho, George Chen, Minh Lam, Martin Li, Nathan Wong, Charlotte Sun, Stefan Lu, Finn McDonald, Ryan Zhang, Leon Zhou, Andy Shan, Daniel Seo and Liam Celinski.\\

Comment from MS: The value of this sort of work should not be underestimated. It is often the case that a result is obtained in research mathematics before it is fully understood conceptually. This article represents one small but crucial step in the direction of understanding this result. I would even say that reaching a complete understanding of why $43$ is the lower bound is equivalent to resolving the conjecture. There are already so many new and interesting possibilities from this vantage point that did not exist before.\\

From an educational perspective, the advantage of having an article like this is clear: it means that Exoo's result can now be taught to students in a meaningful way, which could not have been done before. The nature of the problem makes it particularly appealing to young students.\\


\section{Definition of Exoo's graph}

First, we define an edge colouring on $K_{43}$. Then, we remove one vertex and finally change the colour of some edges.

\subsection{Definition of the colouring for $K_{43}$}

Exoo first defined a colouring of $K_{43}$:

\begin{defn}[Cyclic(43)]
Arrange $43$ points equally spaced around a circle and label them from $0$ to $42$ in clockwise order and in modulo $43$ (so that for all integers $a$, the vertex $a$ is the same as vertex $a+43$). For any two distinct vertices $a, b$ we define the \textit{distance} (denoted by $\dis(a, b)$) between them as the length of the minor arc between the vertices (where we define the minor arc length between adjacent vertices to be $1$).\\

Note that $\dis(a,b)$ can also be defined as the unique integer ${m \in \{ 1, \ldots, 21 \}}$ such that either ${\ModCmp{a-b}{m}{43}}$ or ${\ModCmp{b-a}{m}{43}}$.\\

We will also define the \textit{length} of an edge to be the distance between the two vertices that form the edge. Observe that all edge lengths will be between $1$ and $21$ inclusive. The edges will be coloured depending on their lengths, as shown in the following table, where each length from $1$ to $21$ corresponds to either the colour red or blue.
\vspace{0.35cm}
\begin{center}
\begin{tabular}{|c|c|c|c|c|c|c|c|c|c|c|c|}
\hline
Red & 1 & 2 & 7 & 10 & 12 & 13 & 14 & 16 & 18 & 20 & 21 \\
\hline
Blue & 3 & 4 & 5 & 6 & 8 & 9 & 11 & 15 & 17 & 19 &\\
\hline
\end{tabular}
\end{center}
Denote an edge length as a \textit{red-length} or a \textit{blue-length} if it corresponds to the colour red or blue respectively in the table above.


\vspace{0.2cm}
Since this colouring is defined solely using the lengths of each edge, then this colouring is cyclic over all vertices, so we will denote the graph as $\Cyc(43)$. 
\end{defn}

For illustration purposes, the following diagram shows the colours of the edges of length $1, 7, 11$ of $\Cyc(43)$.

\begin{center}
\includegraphics[scale=0.3]{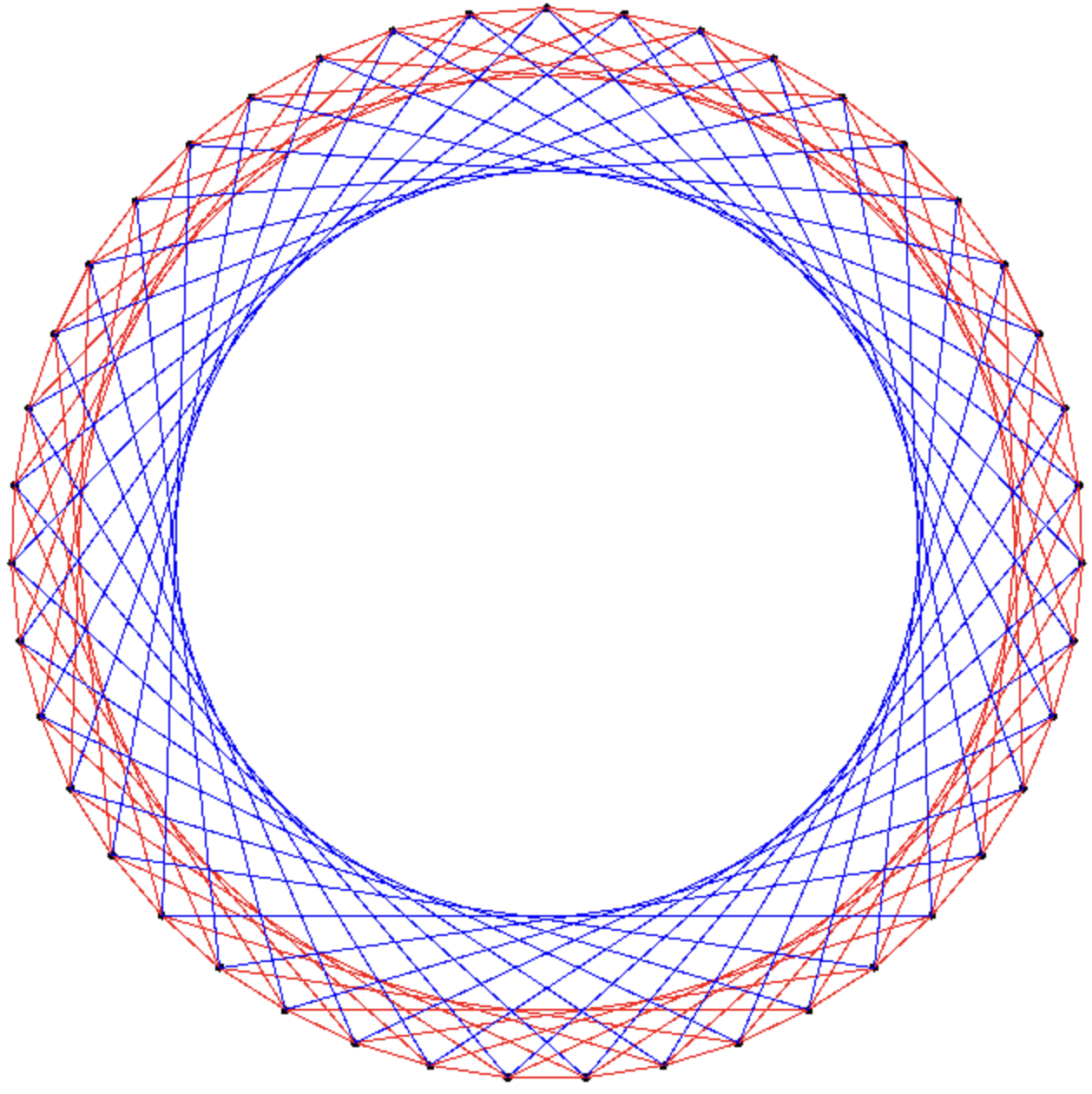}
\end{center}
\vspace{.5cm}
 The following proposition and its converse appears in Exoo's paper \cite{Exoo}. We prove the converse in Section 3 that the only red $K_5$'s in $\Cyc(43)$ are the ones listed below. The proof of the proposition was omitted in the paper because it was trivial but we include it here for completeness.

\begin{prop}\label{2.1} For all integers $i$ such that $0 \leq i \leq 42$, the $5$-tuple of vertices ${(i, i + 1, i + 2, i + 22, i + 23)}$ forms a red $K_5$ in $\Cyc(43)$. This yields $43$ distinct red $K_5$'s.
\end{prop}
\begin{proof}
We will check for all $0 \leq i \leq 42$ that the lengths of the edges between all ten pairs of vertices in each of the $K_5$'s are red-lengths. Below are the edge lengths in the $i=0$ case.
\begin{align*}
&\dis(0, 1) = 1.\\
&\dis(0, 2) = 2.\\
& \dis(0, 22) = 21.\\
&\dis(0, 23) = 20.\\
&\dis(1, 2) = 1.\\
&\dis(1, 22) = 21.\\
&\dis(1, 23) = 21.\\
&\dis(2, 22) = 20.\\
&\dis(2, 23) = 21.\\
&\dis(22, 23) = 1.\\
\end{align*}
Since all the edge lengths are either $1, 2, 20$ or $21$ (which are all red-lengths), then this subgraph is a red $K_5$. By symmetry, the edge lengths of the pairs of vertices will be the same for the $K_5$'s for all other values of $i$, so we obtain a red $K_5$ for every value of $i$.\\

For each value of $i$, since vertex $i+1$ is the only vertex in the corresponding $K_5$ whose adjacent vertices are also in the $K_5$, then this shows that the $K_5$'s for each value of $i$ are unique. Hence, we have $43$ distinct $K_5$'s.

\end{proof}

It is also stated in Exoo's paper that there are no blue $K_5$'s in $\Cyc(43)$. We provide a proof in Section 4.

\subsection{Using $\Cyc(43)$ to colour $K_{42}$}

Here, we recall the definition of Exoo's graph and introduce some notation.

\begin{defn}[$\Exoo(42)$]
Start with $\Cyc(43)$, delete the vertex $0$ (which in turn removes all edges connected to it) and change the colours of the following sixteen edges.

$$\e(4,5) \qquad \e(13,14) \qquad \e(23,24) \qquad \e(39,40)$$
$$\e(5,6) \qquad \e(14,15) \qquad \e(24,25) \qquad \e(40,41)$$
$$\e(6,7) \qquad \e(15,16) \qquad \e(30,31) \qquad \e(41,42)$$
$$\e(7,8) \qquad \e(16,17) \qquad \e(33,34) \qquad \e(11,32)$$
\\
We will denote the resulting graph by $\Exoo(42)$.
\end{defn}
Note that the colour changes are all from red to blue because they involve fifteen edges of the red-length $1$ and one edge of the red-length $21$.\\

The following Lemma is proved in Exoo's paper.
\begin{lem}\label{colourChanges43RedK5Removed} (LC, VY)
The 43 red $K_5$'s of Proposition \ref{2.1} are no longer red $K_5$'s in $\Exoo(42)$ due to the colour changes and the removed vertex $0$.
\end{lem}
\begin{proof}
From Proposition \ref{2.1} we see that each vertex in $\Cyc(43)$ is part of five of the $43$ red $K_5$'s. Thus, only $38$ of these red $K_5$'s will remain after deleting vertex $0$. \\

Recall that the $38$ red $K_5$'s consist of the vertices $\{i, i + 1, i + 2, i + 22, i + 23\}$ for all integers $i$ where $0 \leq i \leq 42$, but excluding the values of ${i=0, 20, 21, 41, 42}$ because vertex $0$ has been deleted. We now show that the colour changes cause all $38$ of the red $K_5$'s to have an edge turn from red to blue. \\

\begin{itemize}
    \item For all $i \in \{4,5,6,7,13,14,15,16,23,24,30,33,39,40\}$, the edge $\e(i,i+1)$ changes colour.
    \item For all $i \in \{3, 12, 22, 29, 32, 38 \}$, the edge $\e(i+1,i+2)$ changes colour.
    \item For all $i \in \{ 1, 2, 8, 11, 17, 18, 19, 25, 26, 27, 28, 34, 35, 36, 37 \}$, the edge ${\e(i+22,i+23)}$ changes colour.
    \item For $i=9$, the edge $\e(i+2, i+23)$ (i.e. $\e(11,32)$) changes colour.
    \item For $i=10$, the edge $\e(i+1, i+22)$ (i.e. $\e(11,32)$) changes colour.
    \item For $i=31$, the edge $\e(i+1, i+23)$ (i.e. $\e(11,32)$) changes colour.
\end{itemize}
Hence, the colour changes and the removed vertex $0$ disrupt all $43$ of the red $K_5$'s from Proposition \ref{2.1}.

\end{proof}


If we combine the above Lemma with a proof that there are no other red $K_5$'s in $\Cyc(43)$ apart from the $43$ previously listed ones, then we can conclude there are no red $K_5$'s in $\Exoo(42)$. This is done in Section 3.

\section{Helpful observations}

\begin{lem} In $\Cyc(43)$ and $\Exoo(42)$, the edge joining vertices $a$ and $b$ is blue if the difference $|a-b|$ is one of the following:
$$3, 4, 5, 6, 8, 9, 11, 15, 17, 19, 24, 26, 28, 32, 34, 35, 37, 38, 39, 40.$$
In $\Cyc(43)$, all other edges are red. In $\Exoo(42)$, all other edges are red unless the difference is $1, 21, 22$ or $42$ in which case the specific edge needs to be checked against the colour change table.
\end{lem}
\begin{proof} The first ten differences are equal to all the blue-lengths. Also, observe that when the latter ten differences are subtracted from $43$, we obtain all the blue-lengths, so these differences (which are distances between pairs of vertices along the major arc) also correspond to blue edges.\\

By definition, all other edges of $\Cyc(43)$ are red. In $\Exoo(42)$, colour changes only occur for certain edges of length $1$ or $21$. This corresponds to the difference $|a-b|$ being either $1, 21, 22$ or $42$ (considering distances along the minor or major arc).

\end{proof}

\begin{defn} Consider a fixed edge $\e(a, b)$ in $\Cyc(43)$, We say that two vertices $u,v$ are \textit{symmetric} with respect to $\e(a, b)$ if 
$$
\ModCmp{a-u}{v-b}{43}
$$
(or equivalently $\ModCmp{a-v}{u-b}{43}$. We say that two sets of vertices $U,V$ of equal size are \textit{symmetric} about $\e(a,b)$ if all the vertices in $U$ can be paired with all the vertices in $V$ such that the vertices in each pair are symmetric with respect to $\e(a,b)$.
\end{defn}

\begin{prop}[Symmetry in Colourings] In $\Cyc(43)$, suppose that we are given a pair of vertices ($a,b$) and another pair of vertices ($x,y$) satisfying
$$\ModCmp{x-b}{a-y}{43}$$
(or in other words vertices $x, y$ are symmetric with respect to $\e(a, b)$). Then the edge $\e(x,a)$ has the same colour as $\e(y,b)$ and the edge $\e(x,b)$ has the same colour as $\e(y,a)$. In particular, if both $\e(x,a)$ and $\e(x, b)$ edges are red (resp. blue), then so are the edges $\e(y,a)$ and $\e(y,b)$ and vice versa.
\end{prop}
\begin{proof}It suffices to show the differences between $x$ and $a,b$ are the same as the differences between $y$ and $b,a$ respectively (in modulo $43$) since equal differences imply that the edge lengths are equal and hence have the same colour.\\

We are already given that ${\ModCmp{x-b}{a-y}{43}}$. By rearranging this, we obtain ${\ModCmp{x-a}{b-y}{43}}$, so we have both of the required equalities in differences.
\end{proof}

This proposition can be applied to a set of vertices $X$ symmetric to a set of vertices $Y$ about a fixed edge $\e(a,b)$. A practical consequence of this symmetry is that it reduces the effort required to determine colours of edges. For example, if the colour of the edge $\e(x, a)$ is known for some $x \in X$, then if $y \in Y$ is the symmetric counterpart of $x$, the colour of $\e(y, b)$ can immediately be deduced to be the same colour as $\e(x,a)$.\\

In the rest of this article, we will use diagrams like the one below. Here, we explain some of our conventions.
\begin{verbatim}

0        1         2         3         4  
1234567890123456789012345678901234567890123  
o  xxxx xx x   x x x    x x x   x xx xxxx 
  o  xxxx xx x   x x x    x x x   x xx xxxx 
     EE E  E       E      E E     E  E EE
      
\end{verbatim}
The first two rows in the diagram represent the $43$ vertices with the tens digit on the first row and the units digit on the second row. \\

Given a fixed colour (blue or red depending on the context), the `o' represents the vertex in question and the x's represent all the vertices in $\Cyc(43)$ which form an edge of the given colour with the vertex corresponding to `o', \\

In the example diagram above, the fixed colour is blue. The third row shows that the vertices which form blue edges with vertex $1$ in $\Cyc(43)$ are vertices 
$${ \{4, 5, 6, 7, 9, 10, 12, 16, 18, 20, 25, 27, 29, 33, 35, 36, 38, 39, 40, 41\} }.$$

The fourth row shows that the vertices in $\Cyc(43)$ which form blue edges with vertex $3$ are vertices 
$${ \{ 6, 7, 8, 9, 11, 12, 14, 18, 20, 22, 27, 29, 31, 35, 37, 38, 40, 41, 42, 43 \}}.$$ 
The E's in the fifth row show overlaps in the x's in the third and fourth rows.\\

One helpful observation is that because we have a cyclic colouring in $\Cyc(43)$, the relative positions of the x's and `o' are always fixed and we simply have to translate a row cyclically to obtain the rows corresponding to different vertices labelled with `o'. Also, the E's give us the information that the vertices that form blue edges with vertices $1$ and $4$ are 
$${ \{ 6, 7, 9, 12, 18, 20, 27, 29, 35, 38, 40, 41 \}}.$$ 
Notice that we can apply our symmetry properties from earlier to see that the first six of these vertices are symmetric with the last six vertices about $\e(1, 3)$. \\

We also see in the rows of the diagram that for the blue edges marked by x's, there are two blocks of four consecutive vertices and two blocks of pairs of consecutive vertices. This leads to the following observation.


\begin{lem}\label{8verticescyclic}Let $a$ be an integer such that $0 \leq a \leq 42$. There are exactly eight vertices which form blue edges with both vertex $a$ and $a+1$ in $\Cyc(43)$. Moreover, four of these vertices are symmetric with the other four over edge $\e(a, a+1)$. These eight vertices have the form
$$a+4,a+5,a+6,a+9, a-8, a-5, a-4, a-3 \ (\bmod\; 43).$$
\end{lem}

\begin{proof} Without loss of generality, we can assume by cyclic symmetry that $a=1$. In the below diagram, the fixed colour is set to blue.

\begin{verbatim}

0        1         2         3         4  
1234567890123456789012345678901234567890123  
o  xxxx xx x   x x x    x x x   x xx xxxx 
 o  xxxx xx x   x x x    x x x   x xx xxxx
    EEE  E                         E  EEE
 
\end{verbatim}
The diagram above shows that the vertices that form blue edges with both vertex $1$ and $2$ are ${ \{ 5, 6, 7, 10, 36, 39, 40, 41 \}}$. These eight vertices are in the form listed in the lemma.\\

To check the symmetry, we observe that the positions of vertices ${ a+4, a+5, a+6, a+9}$ relative to vertex $a+1$ are $3, 4, 5, 8$ steps clockwise respectively and the positions of vertices ${a-3, a-4, a-5, a-8}$ relative to vertex $a$ are $3, 4, 5, 8$ steps anticlockwise respectively.

\end{proof}

\section{Checking red $K_5$'s}
We first prove that the only red $K_5$'s in $\Cyc(43)$ are the $43$ ones claimed by Exoo. By cyclic symmetry, without loss of generality we only have to consider the red $K_5$'s that include vertex $1$. Note that the distances between successive vertices in a red $K_5$ (going clockwise around the circle) add up to $43$ and there are five such distances. Hence, by the pigeonhole principle, at least one of the edge lengths in the red $K_5$ must be at most $8$. \\

Hence, we will also assume without loss of generality that there is a vertex $a$ clockwise of $1$ in the subgraph such that $\dis(a, 1) \leq 8$ and that $\dis(a, 1)$ is the smallest distance between any pair of vertices in the subgraph. Call a red $K_5$ \textit{standard} if it satisfies all these aforementioned conditions.

\begin{lem}\label{no12red}(VY, CK) Suppose that vertices $1$ and $2$ belong to a standard red $K_5$ in $\Cyc(43)$. Then the subgraph must be formed from one of the three $5$-tuples of vertices ${(1,2,3,23,24)}$,  ${(0,1,2,22,23)}$, or  ${(22,23,24,1,2)}$. These subgraphs correspond to $i=1,0,22$ respectively in Proposition \ref{2.1}.
\end{lem}
\begin{proof}
In the following diagrams for this proof, the fixed colour is set to red.
\begin{centerverbatim}

0        1         2         3         4
1234567890123456789012345678901234567890123
oxx    x  x xxx x x xxxx x x xxx x  x    xx
xoxx    x  x xxx x x xxxx x x xxx x  x    x
  E          EE      EEE      EE          E
  
\end{centerverbatim}
The diagram shows that the other three vertices that can be in the standard red $K_5$ must be within 
$${ S = \{ 0, 3, 14, 15, 22, 23, 24, 31, 32 \} }$$ 
(noting that vertex $0$ and $43$ are the same).\\

\textbf{Case 1 (RK): }Vertex $3$ is in the subgraph. 
\begin{centerverbatim}

0        1         2         3         4
1234567890123456789012345678901234567890123
oxx    x  x xxx x x xxxx x x xxx x  x    xx
xoxx    x  x xxx x x xxxx x x xxx x  x    x
xxoxx    x  x xxx x x xxxx x x xxx x  x    
              E       EE       E 
              
\end{centerverbatim}
The above diagram shows that the remaining vertices in the subgraph must be within ${ \{ 15, 23, 24, 32 \} }$. The tuple ${(1,2,3,23,24)}$ forms a possible subgraph (and is listed in the lemma). \\


The distances from vertex $23$ to vertices ${15, 24, 32}$ are $8, 1, 9$ respectively, of which only $1$ is a red-length. Thus, if vertex $23$ is in the subgraph, then so must $24$, which gives the tuple mentioned previously. \\

The distances from vertex $24$ to vertices ${15, 23, 32}$ are $9, 1, 8$ respectively. Similarly to before, if vertex $24$ is in the subgraph, then so must $23$, again giving the same tuple.\\

The only remaining choice is to include vertices $15, 23$ in the subgraph, but this fails because the edge between them has a blue-length of $8$.\\


\textbf{Case 2 (ML): }Vertex $14$ is in the subgraph.
\begin{centerverbatim}

0        1         2         3         4
1234567890123456789012345678901234567890123
oxx    x  x xxx x x xxxx x x xxx x  x    xx
xoxx    x  x xxx x x xxxx x x xxx x  x    x
xx x  x    xxoxx    x  x xxx x x xxxx x x x
              E        E       E          E
              
\end{centerverbatim}
The diagram above shows that the remaining vertices in the subgraph must be within ${ \{ 0, 15, 24, 32 \} }$. Observe that the edge lengths for the pairs of vertices in $\{ 0, 15, 24\}$ are the blue-lengths ${15, 9, 19}$. Thus, at most one of these three vertices can be used and also vertex $32$ must be used.\\

However, the distances from vertex $32$ to vertices ${0, 15, 24}$ are ${11, 17, 8}$ respectively, which are all blue-lengths. Hence, there are no standard red $K_5$'s in this case.\\


\textbf{Case 3 (MLam): }Vertex $15$ is in the subgraph.
\begin{centerverbatim}

0        1         2         3         4
1234567890123456789012345678901234567890123
oxx    x  x xxx x x xxxx x x xxx x  x    xx
xoxx    x  x xxx x x xxxx x x xxx x  x    x
xxx x  x    xxoxx    x  x xxx x x xxxx x x 
  E          E       E        E
              
\end{centerverbatim}
The diagram above shows that the remaining vertices in the subgraph must be within ${ \{ 3, 14, 22, 31 \} }$. However, the distances between each pair of these four vertices are ${11, 19, 15, 8, 17, 9}$ which all all blue-lengths. Hence, there are no standard red $K_5$'s in this case.\\

\textbf{Case 4 (MS): }Vertex $22$ is in the subgraph.
\begin{centerverbatim}

0        1         2         3         4
1234567890123456789012345678901234567890123
oxx    x  x xxx x x xxxx x x xxx x  x    xx
xoxx    x  x xxx x x xxxx x x xxx x  x    x
xx x x xxx x  x    xxoxx    x  x xxx x x xx
              E       EE       E          E
              
\end{centerverbatim}
The diagram above shows that the remaining vertices in the subgraph must be within ${ \{ 0, 15, 23, 24, 32 \} }$. The distances from vertex $0$ to vertices ${15, 23, 24, 32}$ are ${15, 20, 19, 11}$ respectively, of which only $20$ is a red-length, so if vertex $0$ is included, then the subgraph must be formed by the tuple ${(0, 1, 2, 22, 23)}$ (which is mentioned in the lemma).\\

Now assume that vertex $0$ is not included. The distances from vertex $23$ to vertices ${15, 24, 32}$ are ${8, 1, 9}$ respectively, of which only $1$ is a red-length, so if vertex $23$ is included, then the subgraph must be formed by the tuple ${(22, 23, 24, 1, 2)}$ (which is mentioned in the lemma).\\

Now assume that vertices $0, 23$ are not included. Then we must choose two vertices from ${ \{ 15, 24, 32 \}}$. Since the distances between the pairs of these three vertices are ${9,17,8}$ which are blue-lengths, then there are no standard red $K_5$'s in this case.\\

\textbf{Case 5: }Vertex $23$ is in the subgraph. 
\begin{centerverbatim}

0        1         2         3         4
1234567890123456789012345678901234567890123
oxx    x  x xxx x x xxxx x x xxx x  x    xx
xoxx    x  x xxx x x xxxx x x xxx x  x    x
xxx x x xxx x  x    xxoxx    x  x xxx x x x
  E                  E E                  E
              
\end{centerverbatim}
The diagram above shows that the remaining vertices in the subgraph must be within ${ \{ 0, 3, 22, 24 \} }$. Note that we have already covered the possibility of vertices $3$ or $22$ being included earlier, so then we only have to consider the possibility of the remaining two vertices being $0$ and $24$, which fails because the distance between them is $19$, which is a blue-length.\\



\textbf{Case 6: }At least one of the vertices $0, 24, 31, 32$ are in the subgraph. Note that the vertices $0, 24, 31, 32$ are symmetric to the vertices $3, 22, 15, 14$ respectively about $\e(1,2)$. Thus, by symmetry the arguments used in Cases 1, 2, 3 and 4 can be applied. \\

In these cases the only valid standard $K_5$'s found were those formed from the tuples ${(1,2,3,23,24)}$, ${(0, 1, 2, 22, 23)}$ and ${(22, 23, 24, 1, 2)}$. Therefore, the only standard red $K_5$'s in Case 6 must be formed by considering the symmetrical subgraphs of these three tuples about $\e(1,2)$, which gives us the tuples ${(0, 1, 2, 22, 23)}$, ${(1, 2, 3, 23, 24)}$ and ${(22, 23, 24, 1, 2)}$ respectively, so we do not obtain any more new standard red $K_5$'s. \\

Hence, we have covered all cases, where we showed that the only standard red $K_5$'s are formed by the three tuples listed in the lemma.
\end{proof}

\begin{lem}\label{no13red}(AQ,GC) There are no standard red $K_5$'s in $\Cyc(43)$ that contain vertices $1$ and $3$ but not $2$.
\end{lem}
\begin{proof}[Proof by AQ]

\begin{verbatim}

0        1         2         3         4
1234567890123456789012345678901234567890123
oxx    x  x xxx x x xxxx x x xxx x  x    xx 
xxoxx    x  x xxx x x xxxx x x xxx x  x     
 E          E E E E E EE E E E E E        
 
\end{verbatim}

Suppose for the sake of contradiction that there exists a standard red $K_5$ that satisfies the conditions. The above diagram (with fixed colour set to red) shows that apart from the vertex $2$, the vertices that are connected to both $1$ and $3$ by red edges are
$${  S = \{ 13, 15, 17, 19, 21, 23, 24, 26, 28, 30, 32, 34 \}  }.$$
We will partition $S$ into the two sets
\begin{align*}
    T &= \{ 13, 15, 17, 19, 21, 23 \}.\\
    U &= \{ 24, 26, 28, 30, 32, 34 \}.
\end{align*}
Let $V, W$ be the sets of distances between pairs of vertices in $T$ and $U$ respectively. We obtain
$$
V = W = \{ 2, 4, 6, 8, 10 \}.
$$
Note that the only red-lengths in $V, W$ are $2$ and $10$. Suppose that the subgraph contains three vertices ${a, b, c \in T}$ such that ${a < b < c}$. Then $b-a$ and $c-b$ must be either $2$ or $10$, but then the edge length
$$
\dis(a, c) = c-a = (c-b) + (b-a)
$$
would have be either $4, 12$ or $20$, which is a contradiction. Hence, the subgraph must contain at most two vertices from $T$ and using a similar argument, the same must be true for $U$. This forces one of $T, U$ to contain exactly two vertices of the subgraph and the other to contain exactly one vertex. \\

Also, observe that $T, U$ are symmetrical about $\e(1, 3)$, so we can assume without loss of generality that $T$ contains exactly two vertices $a, b$ where $a<b$ and $U$ contains exactly one vertex $c$ from the subgraph.\\

From earlier, we know that $b-a$ must be either $2$ or $10$. Also, note that 
$$
1 \leq c - b < c-a \leq 34-13=21.
$$
Hence, $\dis(c, a) = c-a$ and $\dis(c, b) = c-b$. We therefore obtain
$$
\dis(c, a) - \dis(c, b) = (c-a) - (c-b) = b-a = \dis(a, b) = 2 \text{ or } 10.
$$
By observing the values in $T$ and $U$, notice that $\dis(c, a)$ and $\dis(c, b)$ must both be odd. However, the colour table for $\Cyc(43)$ shows that there are no two odd red-lengths that differ by $2$ or $10$, contradiction.





\end{proof}

\begin{lem}\label{no18red}There are standard no red $K_5$'s in $\Cyc(43)$ that contain vertices $1$ and $8$ and not the vertices $2$ or $3$.

\end{lem}
\begin{proof}
\begin{verbatim}

0        1         2         3         4
1234567890123456789012345678901234567890123
oxx    x  x xxx x x xxxx x x xxx x  x    xx
x    xxoxx    x  x xxx x x xxxx x x xxx x      
              E     EE E E E EE     E

\end{verbatim}
Suppose for the sake of contradiction that there exists a subgraph that satisfies the given conditions. The only vertices between $2$ and $7$ inclusive that form red edges with vertex $1$ in $\Cyc(43)$ are $2, 3$. But these vertices cannot be included in this subgraph. From the ``standard'' property of the subgraph, the minimum distance between any pair of vertices in the subgraph must be $7$.\\

The diagram above (with fixed colour set to red) shows that the remaining vertices in the subgraph must be within 
$$
S = \{ 15, 21, 22, 24, 26, 28, 30, 31, 37 \}.
$$
Suppose that vertex $26$ is in the subgraph. Then the minimum distance criteria gives us that the vertices ${  \{ 21, 22, 24, 28, 30, 31\} }$ cannot be in the subgraph. But then the only vertices left are $15$ and $37$, so there are not enough vertices to form the subgraph, contradiction. Hence, the subgraph cannot contain the vertex $26$.\\

We can partition the elements of $S$ apart from $26$ into the sets
\begin{align*}
T &= \{ 15, 21, 22, 24\}.\\
U &= \{ 28, 30, 31, 37\}.
\end{align*}
Notice that $T, U$ are symmetrical about $\e(1, 8)$, so we can can assume without loss of generality that the subgraph contains at least two vertices $a, b$ from $T$ where $a < b$. The minimum distance criteria means that at most one of vertices $21, 22, 24$ can be in the subgraph, meaning that we must have $a = 15$. Out of the vertices ${21, 22, 24}$, the vertex $15$ only forms a red edge with $22$, so we must have $b=22$. \\

However, $b$ does not form any red edges with vertices in $U$ because its distances from these vertices are ${6, 8, 9, 15}$ respectively, which are blue-lengths. We have reached a contradiction.



\end{proof}

\begin{lem}\label{allStandardRedK5s}
The only standard red $K_5$'s are formed from the tuples of vertices ${(1,2,3,23,24)}$,  ${(0,1,2,22,23)}$, or  ${(22,23,24,1,2)}$. These subgraphs correspond to $i=1,0,22$ respectively in Proposition \ref{2.1}.
\end{lem}
\begin{proof}
By definition, a standard red $K_5$ must contain the vertex $1$ and at least one of the vertices between $2$ and $9$ inclusive. However, the only vertices in this range that form red edges with vertex $1$ are vertices $2, 3, 8$. \\

These cases are all covered in Lemmas \ref{no12red}, \ref{no13red} and \ref{no18red}, and we only obtain the three standard red $K_5$'s listed above.
\end{proof}

\begin{lem}\label{Exactly43RedK5s}
There are exactly $43$ red $K_5$'s in $\Cyc(43)$. These are of the form given in Proposition \ref{2.1}.
\end{lem}
\begin{proof}
Let $A$ be the set of the $43$ red $K_5$'s of the form given in Proposition \ref{2.1}. From Lemma \ref{allStandardRedK5s} we know that there are exactly three standard red $K_5$'s and these are all in $A$. Hence, all red $K_5$'s must be symmetries of these three subgraphs. It suffices to show $A$ are closed under all possible rotations or reflections of the vertices in $\Cyc(43)$ arranged around the circle.\\

Closure under rotations is clear from the way the $A$ is defined. For closure under reflections it suffices to show that the red $K_5$ with vertices ${(0,1,2,22,23)}$, when reflected about the axis through vertex $0$ and the centre of the circle, creates a red $K_5$ that is still in $A$. This is sufficient because we can use additional rotations to show closure of all elements of $A$ under all possible reflections.\\

The subgraph that results from this reflection is the red $K_5$ with vertices ${(20, 21, 41, 42, 0)}$, which is in $A$ because it corresponds to $i=41$ in Proposition \ref{2.1}.\\

Hence, we obtain that there are no red $K_5$'s that are outside of $A$, so this completes the proof.
\end{proof}

\begin{thm} There are no red $K_5$'s in $\Exoo(42)$.
\end{thm}
\begin{proof}
This follows from Lemmas \ref{colourChanges43RedK5Removed} and \ref{Exactly43RedK5s}.



\end{proof}

\section{Checking blue $K_5$'s in $\Cyc(43)$}

In this section we prove that there are no blue $K_5$'s in $\Cyc(43)$. In all the following diagrams, the fixed colour is set to be blue.\\

Suppose that for the sake of contradiction there exists a blue $K_5$ in $\Cyc(43)$ and assume without loss of generality that the vertex $1$ is part of the blue $K_5$.\\

Using the same pigeonhole principle argument as earlier (for the red $K_5$'s), we can assume that there is a vertex $a$ clockwise of $1$ in the subgraph such that $\dis(a,1) \leq 8$ and $\dis(a, 1)$ is the smallest distance between any pair of vertices in the subgraph. In a similar way as earlier,  a blue $K_5$ \textit{standard} if it satisfies these aforementioned conditions.\\

Using the table of red and blue-lengths, the other four vertices in the graph (in order to form blue edges with vertex $1$), must come from the set
$$
S = \{ 4, 5, 6, 7, 9, 10, 12, 16, 18, 20, 25, 27, 29, 33, 35, 36, 38, 39, 40, 41 \}.
$$
By symmetry, it suffices to prove that there are no standard blue $K_5$'s in $\Cyc(43)$. Suppose for the sake of contradiction that a standard blue $K_5$ exists. \\

By the minimum distance condition, the subgraph must contain a vertex between $2$ and $9$ inclusive. Of these vertices, the only ones that form blue edges with vertex $1$ are vertices ${4, 5, 6, 7, 9}$.\\

\textbf{Case 1: }The standard blue $K_5$ contains vertex $4$.\\

\begin{verbatim}

0        1         2         3         4  
1234567890123456789012345678901234567890123  
o  xxxx xx x   x x x    x x x   x xx xxxx 
x  o  xxxx xx x   x x x    x x x   x xx xxx 
xxxx  o  xxxx xx x   x x x    x x x   x xx  
x xxxx  o  xxxx xx x   x x x    x x x   x x
xx xxxx  o  xxxx xx x   x x x    x x x   x  
      
\end{verbatim}
The overlaps in the third and fourth rows of the diagram above show that the vertices that form blue edges with vertices $1$ and $4$ are 
$$
K = \{ 7, 9, 10, 12, 36, 38, 39, 41 \}.
$$
We can partition $K$ into the sets
\begin{align*}
L &= \{ 7, 9, 10, 12 \}.\\
M &= \{ 36, 38, 39, 41 \}.
\end{align*}
Notice that $L, M$ are symmetrical about $\e(1,4)$ so we can assume without loss of generality that the subgraph contains at least two vertices from $L$.\\

Suppose that vertex $7$ is chosen. Then the overlaps in the third, fourth and fifth rows in the diagram show that the only vertices that form blue edges with vertices ${1, 4, 7}$ are vertices ${10, 12, 39, 41}$. \\

Since the distances between pairs of these four vertices are the red-lengths ${ \{2, 12, 14, 16\} }$, then this shows that vertex $7$ cannot be chosen.\\

Suppose that vertex $9$ is chosen but not $7$. Then the overlaps in the third, fourth and sixth rows in the diagram show that the only vertices which form blue edges with vertices ${1, 4, 7}$ are vertices $12$ and $41$. Both of these must be included in order for the subgraph to have five vertices. However the distance between them is the red-length $14$, so this case also fails.\\

Thus, the only way for at least two vertices in the subgraph to be from $L$ is to include vertices $10, 12$. However, the distance between them is the red-length $2$, so this entire case fails.\\

\textbf{Case 2: }The standard blue $K_5$ contains vertex $5$ but not $4$.\\

\begin{verbatim}

0        1         2         3         4  
1234567890123456789012345678901234567890123  
o  xxxx xx x   x x x    x x x   x xx xxxx
xx  o  xxxx xx x   x x x    x x x   x xx xx
        EE     E   E        E   E     EE     
        
\end{verbatim}
The above diagram shows that the vertices that form blue edges with vertices $1$ and $5$ are 
$$
K = \{ 9, 10, 16, 20, 29, 33, 39, 40 \}.
$$
We can partition $K$ into the sets 
\begin{align*}
L &= \{ 9, 10, 16, 20 \}.\\
M &= \{ 29, 33, 39, 40 \}.
\end{align*}
Notice that $L, M$ are symmetrical about $\e(1, 5)$ so we can assume without loss of generality that the subgraph contains at least two vertices from $L$. The only pairs of vertices in $L$ that have blue edges between them are ${(9, 20), (10, 16), (16, 20)}$. This also means that we cannot use three vertices from $L$ in the subgraph, so there must be exactly two vertices from $L$ and exactly one vertex from $M$.\\

Suppose that vertices $10, 16$ are used. Then the distances from vertex $10$ to the vertices in $M$ are ${19, 20, 14, 13}$ respectively, of which only $19$ is a blue-length. Thus, the vertex $29$ from $L$ must also be used. However, the distance between vertex $16$ and $29$ is the red-length $13$, so this case fails.\\

Suppose that vertices $9, 20$ are used. Then the distances from vertex $9$ to the vertices in $M$ are ${20, 19, 13, 12}$ respectively, of which only $19$ is a blue-length. Thus, the vertex $33$ from $L$ must also be used. However, the distance between vertex $20$ and $33$ is the red-length $13$, so this case fails.\\

Suppose that vertices $16, 20$ are used. Then the distances from vertex $16$ to the vertices in $M$ are ${13, 17, 20, 19}$ respectively, of which only $17, 19$ are blue-lengths. Thus, either the vertex $33$ or $40$ from $L$ must be used. However, the distances from vertex $20$ to vertices $33, 40$ are the red-lengths $13, 20$, so this case fails.\\

Hence, there are no standard blue $K_5$'s in this case.\\

\textbf{Case 3: }The standard blue $K_5$ contains vertex $6$ but not $5$ or $4$.\\
\begin{verbatim}

0        1         2         3         4  
1234567890123456789012345678901234567890123  
o  xxxx xx x   x x x    x x x   x xx xxxx
xxx  o  xxxx xx x   x x x    x x x   x xx x
        EE E            E            E EE
        
\end{verbatim}
The above diagram shows that the vertices that form blue edges with vertices $1$ and $6$ are 
$$
K = \{ 9, 10, 12, 25, 38, 40, 41\}.
$$
In order for the subgraph to satisfy the ``standard'' criterion, there can be no vertex of distance less than $5$ to vertex $1$ or $6$ (since vertices $4, 5$ are not used). Thus, the vertices ${9, 10, 40, 41}$ cannot be used, so we are left with vertices ${12, 25, 38}$, where all three of these must be used. However, the distance between vertices $12$ and $25$ is the red-length $13$, so this case fails.\\

\textbf{Case 4: }The standard blue $K_5$ contains vertex $7$ but not $6, 5, 4$.\\
\begin{verbatim}

0        1         2         3         4  
1234567890123456789012345678901234567890123  
o  xxxx xx x   x x x    x x x   x xx xxxx
xxxx  o  xxxx xx x   x x x    x x x   x xx 
   E     E E   E E              E E   E E       
\end{verbatim}
The above diagram shows that the vertices that form blue edges with vertices $1$ and $7$ are 
$$
K = \{ 4, 10, 12, 16, 18, 33, 35, 39, 41\}.
$$
In order for the subgraph to satisfy the ``standard'' criterion, there can be no vertex of distance less than $6$ to vertex $1$ or $7$ (since vertices $4, 5, 6$ are not used). Thus, the vertices ${4, 10, 12, 39, 41}$ cannot be used, so we are left with vertices ${16, 18, 33, 35}$, where three of them must be used. However, since the pairs of vertices $(16, 18)$ and $(33, 35)$ have edges of the red-length $2$ between them, then at least two of the four vertices cannot be used, so this case fails.\\

\textbf{Case 5: }The standard blue $K_5$ contains vertex $9$ but not $7, 6, 5, 4$.\\
\begin{verbatim}

0        1         2         3         4  
1234567890123456789012345678901234567890123  
o  xxxx xx x   x x x    x x x   x xx xxxx
x xxxx  o  xxxx xx x   x x x    x x x   x x 
   EEE     E     E E            E E     E        
\end{verbatim}
The above diagram shows that the vertices that form blue edges with vertices $1$ and $9$ are 
$$
K = \{ 4, 5, 6, 12, 18, 20, 33, 35, 41 \}.
$$
In order for the subgraph to satisfy the ``standard'' criterion, there can be no vertex of distance less than $8$ to vertex $1$ or $9$ (since vertices $4, 5, 6, 7$ are not used). Thus, the vertices ${4, 5, 6, 12, 41}$ cannot be used, so we are left with vertices ${18, 20, 33, 35}$, where three of them must be used. However, since the pairs of vertices $(18, 20)$ and $(33, 35)$ have edges of the red-length $2$ between them, then at least two of the four vertices cannot be used, so this case fails.\\

Hence, by exhausting all cases, we have proven that there are no standard blue $K_5$'s in $\Cyc(43)$, which is sufficient to prove that there are no blue $K_5$'s in $\Cyc(43)$.

\section{Checking blue $K_5$'s in $\Exoo(42)$}

Since we have proven in the previous section that there are no blue $K_5$'s in $\Cyc(43)$, in order to prove the same for $\Exoo(42)$, it suffices to only consider $K_5$'s that contain at least one of the edges involved in the colour changes. We use a general argument when considering $K_5$'s that contain at least one of the fifteen edges between consecutive vertices that are involved in colour changes and then afterwards consider the edge $\e(11, 32)$.\\



We first give one example of our general argument for the fifteen blue colour-changed edges involving consecutive vertices.

\pagebreak

\begin{eg}
The blue edge $\e(4, 5)$ is not part of a blue $K_5$ in $\Exoo(42)$.
\end{eg} 

\begin{proof}
{
\begin{verbatim}

0        1         2         3         4   
1234567890123456789012345678901234567890123   
x  oX xxxx xx x   x x x    x x x   x xx xxx 
xx XoX xxxx xx x   x x x    x x x   x xx xx
E      EEE  E                         E  EE
       
\end{verbatim}
}
Suppose for the sake of contradiction that there exists a blue $K_5$ that includes $\e(4, 5)$. In the above diagram (where the fixed colour is blue and we include vertex $0$), the uppercase `X' represents the edge $\e(4, 5)$ that underwent a colour change. We see that the vertices that form blue edges with vertices $4$ and $5$ are
$$
K = \{ 8, 9, 10, 13, 39, 42, 0, 1 \}.
$$
No pair of these vertices are involved in colour changes. We can partition $K$ into the sets 
\begin{align*}
L &= \{ 8, 9, 10, 13 \}.\\
M &= \{ 39, 42, 0, 1 \}.
\end{align*}
Notice that $L, M$ are symmetrical about $\e(4, 5)$ so we can assume without loss of generality that the subgraph contains at least two vertices from $L$ (there are no colour changes here that affect this assumption). Since the distances $1$ and $2$ are red-lengths, then at most one vertex from each of $\{8, 9, 10\}$ and  $\{42, 0, 1\}$ can be used. This forces exactly two vertices from $L$ and exactly one vertex from $M$ to be in the subgraph.\\

Also, vertex $13$ from $L$ must be used. The distances from vertex $13$ to each vertex in $M$ are ${17, 14, 13, 12}$ respectively, of which only $17$ is a blue-length. Therefore, only vertex $39$ can be used from $M$. However, a similar argument using symmetry shows that the only vertex from $L$ that can be used together with $39$ is vertex $13$, so then we cannot pick a fifth vertex in the subgraph, contradiction.







       

\end{proof}

To generalise the idea in the previous proof, we will show that given any of the fifteen colour-changed edges between adjacent vertices $a, a+1$ in $\Exoo(42)$, there are always exactly eight vertices (when vertex $0$ is included) that form blue edges with $a$ and $a+1$ where four of the vertices are symmetrical with the other four about $\e(a, a+1)$.\\


We also note here that had, say $\e(0, 1)$ underwent a colour change but vertex $0$ had not been deleted, there would be the following blue $K_5$'s formed by vertices  ${(0,1,4,5,39)}$ and ${(0,1,4,5,9)}$.

\begin{lem}\label{8vertices}
For all $a\in\{4,5,6,7,13,14,15,16,23,24,30,33,39,40,41\}$, there are exactly eight vertices $b$ in $\Exoo(42)$ (when we allow vertex $0$ to be included) such that edges $\e(a,b)$ and $\e(a+1, b)$ are both blue (note that edge $\e(a, a+1)$ is also blue due to the colour changes). These vertices have the form 
$$a+4,a+5,a+6,a+9, (a-8), (a-5), (a-4), (a-3) \ (\bmod\; 43).$$

No pair of these eight vertices belongs to a colour change and the first four vertices listed are symmetric to the last four about $\e(a, a+1)$.


\end{lem}
\begin{proof} 
First, note that edges $\e(a,a+2)$ and $\e(a-1,a+1)$ must be red since their edge lengths are both the red-length $2$ and also the edges are not involved in colour changes. Thus, the vertices $a-1$ and $a+2$ cannot be included in our list above. Also, since neither $a$ nor $a+1$ can be $11$ or $32$, this means that we can apply Lemma \ref{8verticescyclic} here because the colour changes do not give us any extra vertices $b$ that form blue edges with vertices $a$ and $a+1$ apart from the eight ones listed in that lemma.\\

It now suffices to check that no pair of these eight vertices are involved in a colour change. Since no distance between a pair of vertices is large enough to correspond to edge $\e(11, 32)$, it suffices to check the edges between consecutive vertices. We check these values of $a$ in blocks.\\

\textbf{Case 1 (MS, RK, LZ): } $a\in \{4,5,6,7\}$. For the vertices $a+4, a+5, a+6$, we have that they range from $8$ and $13$ inclusive, which fits in the gap between $\e(7,8)$ and $\e(13,14)$ on the colour change table. For the vertices $a-5, a-4, a-3$ we have that they can only take on the values $42, 0, 1, 2, 3, 4$, which fits in the gap between $\e(41, 42)$ and $\e(4,5)$ on the colour change table.\\

\textbf{Case 2 (YJ, VY): } $a \in \{13, 14, 15, 16\}$. For the vertices $a+4, a+5, a+6$, we have that they range from $17$ and $22$ inclusive, which fits in the gap between $\e(16,17)$ to $\e(23,24)$ on the colour change table.  For the vertices $a-5, a-4, a-3$ we have that they range from $8$ to $13$ inclusive, which fits in the gap between $\e(7,8)$ and $\e(13,14)$ on the colour change table. \\

\textbf{Case 3 (DS): } $a\in\{23,24\}$. For the vertices $a+4,a+5,a+6$ we have that they range from $27$ to $30$ inclusive, which fits in the gap between $\e(24,25)$ and $\e(30,31)$ in the colour change table. For the vertices $a-5, a-4, a-3$ we have that they range from $18$ to $21$ inclusive, which fits in the gap between $\e(16,17)$ and $\e(23,24)$ in the colour change table.\\

\textbf{Case 4 (LC, CK): } $a \in \{30, 33\}$. For the vertices $a+4, a+5, a+6$ we have that they range from $34$ to $39$ inclusive, which fits in the gap between $\e(33,34)$ and $\e(39,40)$ in the colour change table. For the vertices $a-5, a-4, a-3$ we have that they range from $25$ to $30$ inclusive, which fits in the gap between $\e(24,25)$ and $\e(30,31)$ in the colour change table.\\ 

\textbf{Case 5 (RZ, LG): } $a\in \{39, 40, 41\}$. For the vertices $a+4, a+5, a+6$ we have that they range from $0$ to $4$ inclusive, which fits in the gap between $\e(41,42)$ and $\e(4,5)$ in the colour change table. For the vertices $a-5, a-4, a-3$, we have that they range from $34$ to $38$ inclusive, which fits in the gap between $\e(33,34)$ and $\e(39,40)$ in the colour change table.\\

Hence, there are no colour changes involving consecutive pairs of these eight vertices.

\end{proof}

We also note that the absence of colour changes does not depend on the deletion of vertex $0$ and only on the colour change table.

\begin{thm}\label{blueconsec}
The fifteen blue edges of the form $\e(a,a+1)$ in $\Exoo(42)$ that are involved in colour changes, are not part of any blue $K_5$'s.
\end{thm}
\begin{proof}
Suppose for the sake of contradiction that such an edge is part of a blue $K_5$. By Lemma \ref{8vertices}, there are exactly eight vertices that form blue edges both vertex $a$ and $a+1$. We know that no pair of these eight vertices are involved in colour changes.\\

Out of the first four vertices listed in Lemma \ref{8vertices}, the vertices ${a+4, a+5, a+6}$ are consecutive. In the absence of colour changes, since the distances of $1, 2$ are red-lengths, then at most one of these three vertices can be in the subgraph.\\

Since the distance from vertex $a+9$ to the last four vertices are ${17, 14, 13, 12}$ respectively, of which only $17$ is a blue-length, then vertex $a+9$ only forms a blue edge with $a-8$ out of the last four vertices. \\

By symmetry, the vertex $a-8$ only forms a blue edge with $a+9$ out of the first four vertices. Thus, if either vertex $a+9$ or $a-8$ is part of the subgraph, then both must be present because otherwise we would not have enough vertices in the subgraph. If neither are in the subgraph, then we again have insufficient vertices because we can only use one vertex from each of the two blocks of three consecutive vertices.\\

Hence, we have reached a contradiction.


\end{proof}



 

\begin{prop}\label{blue1132} The blue edge $\e(11, 32)$ is not part of any blue $K_5$ in $\Exoo(42)$.
\end{prop}

\begin{proof}[Proof by AQ]\ \\
\begin{verbatim}

0        1         2         3         4   
1234567890123456789012345678901234567890123
 xx xxxx  o  xxxx xx x   x x x X  x x x   x 11
   x x x  X x x x   x xx xxxx  o  xxxx xx x 32
     E E      E E        E E      E E     E
 
\end{verbatim}
Suppose for the sake of contradiction that there exists a blue $K_5$ that includes $\e(11, 32)$. In the above diagram (where the fixed colour is blue), the uppercase `X' represents the edge $\e(11, 32)$ that underwent a colour change. We see that the vertices that form blue edges with vertices $11$ and $32$ (apart from the removed vertex $0$) are 
$$
K = \{ 6, 8, 15, 17, 26, 28, 35, 37 \}.
$$
We can partition $K$ into the sets 
\begin{align*}
L &= \{ 6, 8, 15, 17 \}.\\
M &= \{ 26, 28, 35, 37 \}.
\end{align*}
Notice that $L, M$ are symmetrical about $\e(11, 32)$ so we can assume without loss of generality that the subgraph contains at least two vertices from $L$ (there are no colour changes here that affect this assumption).\\

Since ${\dis(6,8) = \dis(15,17)=2}$, which is a red-length, then at most one of the vertices in each of $(6, 8)$ and $(15, 17)$ can be used. Therefore, exactly two vertices in each of $L$ and exactly one from $M$ must be in the subgraph.\\

If vertex $6$ is used, observe that the distances from it to each vertex in $M$ are ${20, 21, 14, 12}$ respectively, which are all red-lengths. Thus, we cannot select any vertex from $M$, so this case fails.\\

If vertex $8$ is used, observe that the distances from it to each vertex in $M$ are ${18, 20, 16, 14}$ respectively, which are all red-lengths. Thus, we cannot select any vertex from $M$, so this case fails.\\

Hence, at most one vertex from $L$ can be selected for the subgraph, contradiction.


\end{proof}

This concludes the proof that there are no blue $K_5$'s in $\Exoo(42)$.


 

\section{Further investigations}
One benefit of the proofs given in this article is that we can gain a greater understanding for the reasons for why Exoo's construction works. Also, the following questions are now readily accessible:

\begin{enumerate}
    \item If vertex $0$ was not deleted from $\Cyc(43)$ but the same colour changes were made, how many red and blue $K_5$'s are there?
    
    \item Is there a colouring of $K_{43}$ with at most ten monochromatic $K_5$'s?
\end{enumerate}

\begin{eg}Define a colouring on $K_{43}$ the same as $\Cyc(43)$ and make the same colour changes as in $\Exoo(42)$ but without deleting vertex $0$. This colouring has exactly four red $K_5$'s represented by the tuples
$$
(0,1,2,22,23), (0,1,21,22,42), (0,1,21,22,23), (0,20,21,22,42)
$$
and nine blue $K_5$'s represented by the tuples
\begin{align*}
&(0,6,11,15,32), (0,6,11,17,32), (0,8,11,17,32), \\
&(0,11,28,32,37), (0,11,26,32,37), (0,11,26,32,35),\\
&(0,11, 15, 26,32), (0, 11, 17, 26, 32), (0, 11, 17, 28, 32).\\
\end{align*}
\end{eg}
\begin{proof}
The $5$-tuples of vertices $(i,i+1,i+2,i+22,i+23)$ for ${i \in \{20,21,41,42,0\}}$ form the five red $K_5$'s in $\Cyc(43)$ that use vertex $0$. Only the subgraph corresponding to $i=41$ is no longer a red $K_5$ after the colour changes (specifically, the colour change of $\e(41, 42)$). \\

\pagebreak 

\begin{verbatim}

0        1         2         3         4   
1234567890123456789012345678901234567890123
 xx xxxx  o  xxxx xx x   x x x X  x x x   x 11
   x x x  X x x x   x xx xxxx  o  xxxx xx x 32
  xxxx xx x   x x x    x x x   x xx xxxx  o            
     E E      E E        E E      E E     
 
\end{verbatim}

For the blue $K_5$'s, observe that our proof of Theorem \ref{blueconsec} did not rely on vertex $0$ being removed, so we only have to search for blue $K_5$'s  which include vertex $0$ and edge $\e(11, 32)$.\\

Using the diagram above and our work and notation used in the proof of Proposition \ref{blue1132}, note that all the vertices in 

$$
K = \{ 6, 8, 15, 17, 26, 28, 35, 37 \}
$$ 
form blue edges with vertex $0$. Hence, we just have to find two additional vertices $K$ that have a blue edge between them. \\

From our earlier work, if vertex $6$ is chosen, then we cannot choose any vertices from $M$, so we must either choose vertex $15$ or $17$ from $M$. This gives us the first two tuples listed, since $\e(6,15)$ and $\e(6,17)$ have lengths $9, 11$ respectively, which are blue-lengths.\\

If vertex $8$ is chosen, we similarly cannot choose any vertices from $M$, so we must either choose vertex $15$ or $17$ from $M$. This gives us the third tuple listed, since $\e(8,15)$ and $\e(8,17)$ have lengths $7, 9$ respectively, where only $9$ is a blue-length.\\

If vertex $37$ is used, since it is symmetrical with vertex $6$ about $\e(11,32)$, we obtain the fourth and fifth tuples listed (which are symmetrical with the first two tuples respectively).\\

If vertex $35$ is used, since it is symmetrical with vertex $8$ about $\e(11,32)$, we obtain the sixth tuple listed (which is symmetrical with the third tuple).\\

If none of these four vertices are used, note that the distances between pairs of vertices in $\{ 15, 17, 26, 28\}$ are one of ${2, 9, 11, 13}$, where only $9$ and $11$ are blue-lengths, which corresponds to the pairs of vertices $(15,26)$, $(17, 26)$, $(17, 28)$. These give us the seventh, eighth and ninth tuples listed respectively.\\


\end{proof}

\begin{eg}Define a colouring on $K_{43}$ the same as $\Cyc(43)$ and make the same colour changes in $\Exoo(42)$ but without deleting vertex $0$ but also change the colour of $\e(21, 22)$ to blue. The resulting graph has one red $K_5$ and eight blue $K_5$'s.
\end{eg}

This leads naturally to the following questions:\\

{\bf Questions:}
\begin{enumerate}
\item If vertex $0$ was not deleted from $\Cyc(43)$ but you are free to make colour changes from red to blue, what is the minimum number of monochromatic $K_5$'s you can achieve?\\
\item What is the fewest number of monochromatic $K_5$'s possible for a colouring of $K_{43}$?\\
\end{enumerate}

If the answer to the last question listed is non-zero then the conjecture is true, otherwise the conjecture is false.\\

If the conjecture is true, then this number would be what is known as the Ramsey multiplicity of $K_5$. See for example \cite{PR}, where the Ramsey multiplicity of $K_4$ is shown to be $9$.

It has been remarked by Exoo in unpublished work that he has come across an example with only two monochromatic $K_5$'s.


{\sc Dr Michael Sun's School of Maths,  Sydney, Australia}

\end{document}